\documentclass[a4paper,11pt]{article}

\usepackage[T1]{fontenc}
\usepackage{lmodern}

\usepackage{dsfont}

\usepackage[latin1]{inputenc}
\usepackage{amsmath}
\usepackage{amsthm}
\usepackage{amssymb}
\usepackage{mathrsfs}
\usepackage{graphicx}
\usepackage[all]{xy}
\usepackage{hyperref}

\usepackage{stmaryrd}
\usepackage{caption}

\usepackage{abstract} 

\newtheorem{thm}{Theorem}[section]
\newtheorem{cor}[thm]{Corollary}
\newtheorem{claim}[thm]{Claim}

\newtheorem{lemma}[thm]{Lemma}
\newtheorem{prop}[thm]{Proposition}

\theoremstyle{definition}
\newtheorem{definition}[thm]{Definition}

\newtheorem{remark}[thm]{Remark}
\newtheorem{question}[thm]{Question}

\title{Rank-one isometries of CAT(0) cube complexes and their centralisers}
\date{\today}
\author{Anthony Genevois}

\begin{document}

\maketitle

\begin{abstract}
If $G$ is a group acting geometrically on a CAT(0) cube complex $X$ and if $g \in G$ is an infinite-order element, we show that exactly one of the following situations occurs: (i) $g$ defines a rank-one isometry of $X$; (ii) the stable centraliser $SC_G(g)= \{ h \in G \mid \exists n \geq 1, [h,g^n]=1 \}$ of $g$ is not virtually cyclic; (iii) $\mathrm{Fix}_Y(g^n)$ is finite for every $n \geq 1$ and the sequence $(\mathrm{Fix}_Y(g^n))$ takes infinitely many values, where $Y$ is a cubical component of the Roller boundary of $X$ which contains an endpoint of an axis of $g$. We also show that (iii) cannot occur in several cases, providing a purely algebraic characterisation of rank-one isometries.
\end{abstract}

\tableofcontents

\section{Introduction}

\noindent
A major progress in the study of the geometry of CAT(0) cube complexes made in the last years was the proof of the rank rigidity conjecture \cite{CapraceSageev}. Namely, if a group $G$ acts essentially and without fixed point at infinity on a CAT(0) cube complex $X$, then either $G$ contains a rank-one isometry or $X$ decomposes as a product. Among the applications, let us mention: constructions of free subgroups and variations \cite{CapraceSageev, PingPong, CubeUniformGrowth}, acylindrical hyperbolicity (see \cite[Section 6.2]{MoiHypCube} and references therein), random walks \cite{CCCrandom, random} and obstructions to act on CAT(0) cube complexes \cite{CIF, ThompsonFW, CubeKahler, TorsionCube}. All these applications motivate the fundamental role played by rank-one isometries in the geometry of CAT(0) cube complexes.

\medskip \noindent
In this article, we are interested in the following question:

\begin{question}
Let $G$ be a group acting geometrically on a CAT(0) cube complex $X$. Does there exist a purely algebraic characterisation of the elements of $G$ which induce rank-one isometries of $X$?
\end{question}

\noindent
A natural attempt is to ask the \emph{stable centraliser} 
$$SC_G(g)= \{ h \in G \mid \exists n \geq 1, [h,g^n]=1 \}$$
of an element $g$ of our group $G$ to be virtually cyclic. (Notice that $SC_G(g)$ contains the centraliser of $g$ and has finite index in the commensurator of $\langle g \rangle$.) Unfortunately, it turns out that the stable centraliser may be virtually cyclic while the isometry is not rank-one. More precisely, \cite[Corollary 21]{Rattaggi} provides the example of a commutative-transitive group $G$ acting geometrically on a product of two trees $T_1 \times T_2$ and containing an infinite-order element $g$ whose (stable) centraliser is infinite cyclic.

\medskip \noindent
The main goal of the article is to understand how the equivalence between being a rank-one isometry and having a virtually cyclic stable centraliser may fail. In this context, we prove the following statement:

\begin{thm}\label{intro:Main}
Let $G$ be a group acting geometrically on a CAT(0) cube complex $X$, and $g \in G$ an infinite-order element. Fix a cubical component $Y \subset \mathfrak{R}X$ which contains an endpoint of an axis of $g$. Then exactly one of the following situations occurs:
\begin{itemize}
	\item $g$ defines a rank-one isometry of $X$;
	\item the stable centraliser $SC_G(g)$ of $g$ is not virtually cyclic;
	\item $\mathrm{Fix}_Y(g^n)$ is finite for every $n \geq 1$ and the sequence $(\mathrm{Fix}_Y(g^n))$ takes infinitely many values. 
\end{itemize}
\end{thm}

\noindent
The third case of our trichotomy is precisely what happens in Rattaggi's example. More precisely, our isometry $g \in \mathrm{Isom}(T_1 \times T_2)$ stabilises $T_1$, so that a cubical component which contains an endpoint of an axis of $g$ must be an unbounded and locally finite tree $T$ (actually, $T$ must be isometric to $T_2$). Then $g$ induces an isometry $h$ of $T$ such that the fixed-set of $h^n$ increases as $n \to + \infty$ but always remains bounded. 

\medskip \noindent
However, such a behavior seems to be exotic. In most cases, it typically does not happen. In this context, we deduce from Theorem \ref{intro:Main} the following statement:

\begin{thm}\label{intro:Special}
Let $G$ be a group acting geometrically on a CAT(0) cube complex $X$. Assume that, for every hyperplane $J$ of $X$ and every element $g \in G$, the two hyperplanes $J$ and $gJ$ cannot be transverse nor tangent. Then an infinite-order element of $G$ defines a rank-one isometry of $X$ if and only if its stable centraliser in $G$ is virtually cyclic.
\end{thm}

\noindent
For instance, our theorem applies to Haglund and Wise's (virtually) cocompact special groups \cite{Special}, which include many groups of interest such as right-angled Artin groups, graph products of finite groups (e.g., right-angled Coxeter groups) or graph braid groups. 

\medskip \noindent
In our general study of stable centralisers of isometries, we also prove the following statement, which we think to be of independent interest:

\begin{thm}\label{intro:Regular}
Let $G$ be a group acting geometrically on a CAT(0) cube complex $X$. Assume that $G$ decomposes as a product of $n \geq 1$ unbounded irreducible CAT(0) cube complexes $X_1 \times \cdots X_n$. If $g \in G$ is a regular element, then $SC_G(g)$ is virtually $\mathbb{Z}^n$. 
\end{thm}

\noindent
Recall from \cite{CapraceSageev} that $g$ is \emph{regular} if it induces a rank-one isometry on each factor $X_i$. The existence of regular elements has been proved in \cite{CCCrandom} using probabilistic methods (see also \cite[Theorem 6.67]{MoiHypCube} for an alternative proof based on cubical and hyperbolic geometries). 

\medskip \noindent
The article is organised as follows. In Section \ref{section:Pre}, we recall general definitions and record preliminary statements for future use. In Section \ref{section:SMin}, we introduce \emph{stable minimising sets} and following \cite{FTT} we prove a decomposition theorem. The connection between the stable minimising set and the property of being a rank-one isometry is explained in Section \ref{section:Contracting}, and Theorem \ref{intro:Main} is proved in Section \ref{section:Main}. Finally, a few applications are proved in Section \ref{section:Applications}, including Theorems \ref{intro:Special} and \ref{intro:Regular} above, and we conclude the article with open questions in Section \ref{section:Questions}.

\paragraph{Acknowledgments.} This work was supported by a public grant as part of the Fondation Math\'ematique Jacques Hadamard.

\section{Preliminaries}\label{section:Pre}

\subsection{Cube complexes, hyperplanes, projections}

\noindent
A \textit{cube complex} is a CW complex constructed by gluing together cubes of arbitrary (finite) dimension by isometries along their faces. It is \textit{nonpositively curved} if the link of any of its vertices is a simplicial \textit{flag} complex (ie., $n+1$ vertices span a $n$-simplex if and only if they are pairwise adjacent), and \textit{CAT(0)} if it is nonpositively curved and simply-connected. See \cite[page 111]{MR1744486} for more information.

\medskip \noindent
Fundamental tools when studying CAT(0) cube complexes are \emph{hyperplanes}. Formally, a \textit{hyperplane} $J$ is an equivalence class of edges with respect to the transitive closure of the relation identifying two parallel edges of a square. Notice that a hyperplane is uniquely determined by one of its edges, so if $e \in J$ we say that $J$ is the \textit{hyperplane dual to $e$}. Geometrically, a hyperplane $J$ is rather thought of as the union of the \textit{midcubes} transverse to the edges belonging to $J$ (sometimes referred to as its \emph{geometric realisation}). See Figure \ref{figure27}. The \textit{carrier} $N(J)$ of a hyperplane $J$ is the union of the cubes intersecting (the geometric realisation of) $J$. 

\medskip \noindent
There exist several metrics naturally defined on a CAT(0) cube complex. In this article, we are only interested in the graph metric defined on its one-skeleton, referred to as its \emph{combinatorial metric}. In fact, from now on, we will identify a CAT(0) cube complex with its one-skeleton, thought of as a collection of vertices endowed with a relation of adjacency. In particular, when writing $x \in X$, we always mean that $x$ is a vertex of $X$. 

\medskip \noindent
The following theorem will be often used along the article without mentioning it.

\begin{thm}\emph{\cite{MR1347406}}
Let $X$ be a CAT(0) cube complex.
\begin{itemize}
	\item If $J$ is a hyperplane of $X$, the graph $X \backslash \backslash J$ obtained from $X$ by removing the (interiors of the) edges of $J$ contains two connected components. They are convex subgraphs of $X$, referred to as the \emph{halfspaces} delimited by $J$.
	\item A path in $X$ is a geodesic if and only if it crosses each hyperplane at most once.
	\item For every $x,y \in X$, the distance between $x$ and $y$ coincides with number of hyperplanes separating them.
\end{itemize}
\end{thm}
\begin{figure}
\begin{center}
\includegraphics[trim={0 13cm 10cm 0},clip,scale=0.4]{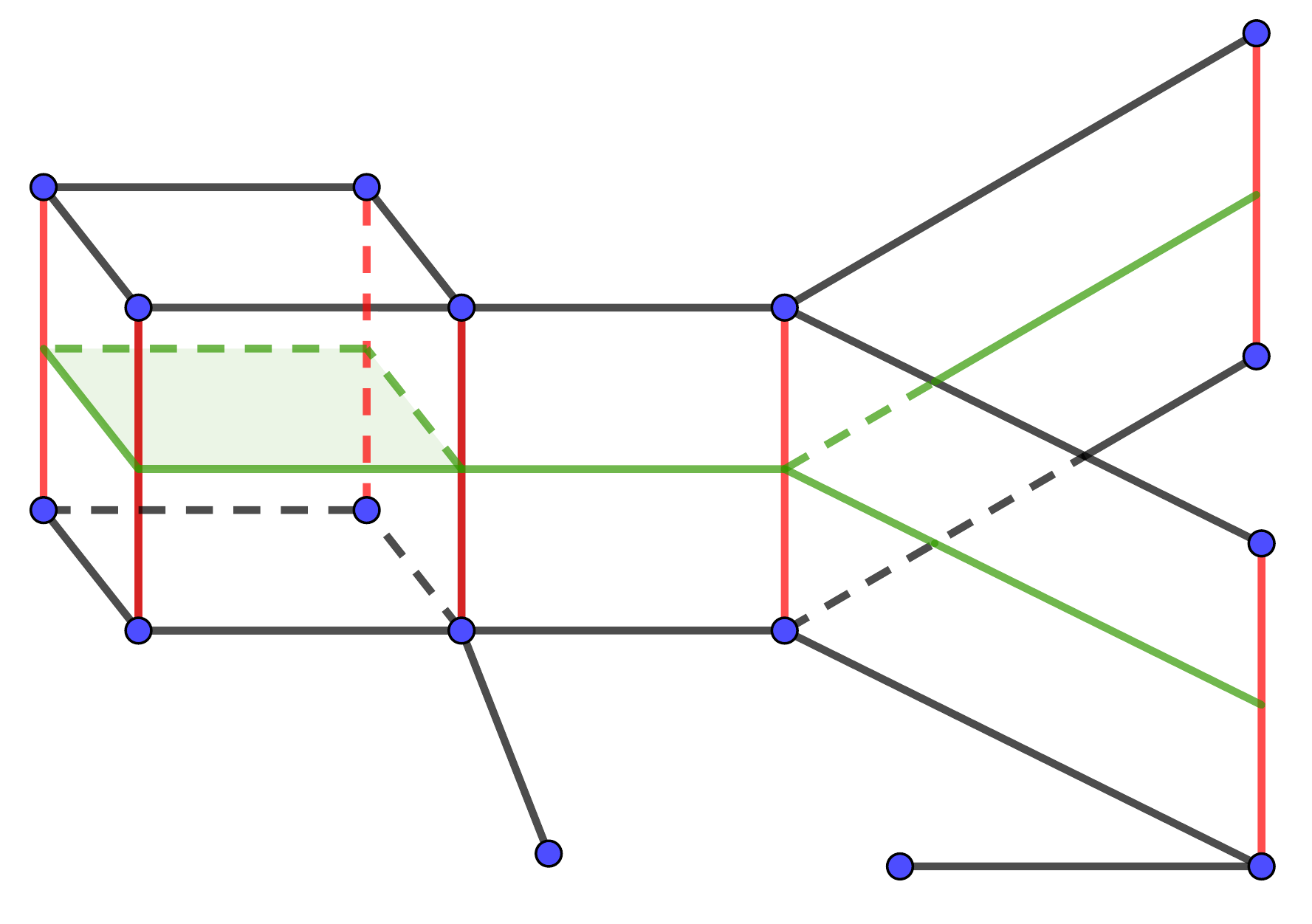}
\caption{A hyperplane (in red) and the associated union of midcubes (in green).}
\label{figure27}
\end{center}
\end{figure}

\noindent
Another useful tool when studying CAT(0) cube complexes is the notion of \emph{projection} onto on a convex subcomplex, which is defined by the following proposition (see \cite[Lemma 13.8]{Special}):

\begin{prop}\label{projection}
Let $X$ be a CAT(0) cube complex, $C \subset X$ a convex subcomplex and $x \in X \backslash C$ a vertex. There exists a unique vertex $y \in C$ minimizing the distance to $x$. Moreover, for any vertex of $C$, there exists a geodesic from it to $x$ passing through $y$.
\end{prop}

\noindent
Below, we record a couple of statements related to projections for future use. Proofs can be found in \cite[Lemma 13.8]{Special} and \cite[Proposition 2.7]{article3} respectively.

\begin{lemma}\label{lem:HypProj}
Let $X$ be a CAT(0) cube complex, $Y \subset X$ a convex subcomplex and $x \in X$ a vertex. Any hyperplane separating $x$ from its projection onto $Y$ separates $x$ from $Y$. 
\end{lemma}

\begin{lemma}\label{lem:HypProjSeparate}
Let $X$ be a CAT(0) cube complex and $Y \subset X$ a convex subcomplex. For every vertices $x,y \in X$, the hyperplanes separating the projections of $x$ and $y$ onto $Y$ are precisely the hyperplanes separating $x$ and $y$ which cross $Y$. As a consequence the projection onto $Y$ is $1$-Lipschitz.
\end{lemma}

\noindent
Notice that the next statement is a direct consequence of Lemma \ref{lem:HypProj}:

\begin{cor}\label{cor:HypSepTwoConvex}
Let $X$ be a CAT(0) cube complex and $Y_1,Y_2 \subset X$ two disjoint and non-empty convex subcomplexes. If $y_1 \in Y_1$ and $y_2 \in Y_2$ are two vertices minimising the distance between $Y_1$ and $Y_2$, then the hyperplanes separating $y_1$ and $y_2$ are exactly the hyperplanes separating $Y_1$ and $Y_2$. 
\end{cor}

\subsection{Isometries of CAT(0) cube complexes}\label{section:Isom}

\noindent
According to \cite{HaglundAxis}, an isometry $g \in \mathrm{Isom}(X)$ of a CAT(0) cube complex is
\begin{itemize}
	\item a \emph{loxodromic isometry} if there exists a bi-infinite geodesic on which $g$ acts by translations;	
	\item an \emph{elliptic isometry} if $g$ has bounded orbits;
	\item an \emph{inversion} if a power of $g$ stabilises a hyperplane and inverts its halfspaces.
\end{itemize}
It is worth noticing that, up to subdividing the cube complex, we may suppose that inversions do not exist. 

\medskip \noindent
\textbf{Convention:} In this article, we always suppose that a CAT(0) cube complex does not admit inversions.

\medskip \noindent
When studying centralisers, natural subsets to consider are:

\begin{definition}
Let $X$ be a CAT(0) cube complex and $g \in \mathrm{Isom}(X)$ a loxodromic isometry. The \emph{minimising set} of $g$ is
$$\mathrm{Min}(g)= \left\{ x \in X \mid d(x,gx)= \inf\limits_{y \in X} d(y,gy) \right\}.$$
Equivalently, $\mathrm{Min}(g)$ is the union of all the axes of $g$.
\end{definition}

\noindent
(For a proof of the equivalence, we refer to \cite[Corollary 6.2]{HaglundAxis}.)

\medskip \noindent
The interest of minimising sets is justified in particular by the following statement, proved in \cite[Lemma 6.3]{FTT}:

\begin{lemma}\label{lem:centraliserCC}
Let $G$ be a group acting geometrically on a CAT(0) cube complex $X$ and $g \in G$ a loxodromic isometry. The centraliser $C_G(g)$ acts geometrically on $\mathrm{Min}(g)$. 
\end{lemma}

\noindent
Rank-one isometries may be defined in many equivalent ways. Let us mention some of these equivalent definitions.

\begin{prop}\label{prop:contracting}
Let $X$ be a uniformly locally finite CAT(0) cube complex and $g \in \mathrm{Isom}(X)$ an isometry. The following conditions are equivalent:
\begin{itemize}
	\item[(i)] $g$ is a \emph{rank-one isometry}, i.e., no CAT(0)-axis of $g$ bounds a CAT(0)-halfplane;
	\item[(ii)] $g$ is contracting with respect to the CAT(0) metric;
	\item[(iii)] $g$ is a Morse isometry with respect to the CAT(0) metric;
	\item[(iv)] $g$ is contracting with respect to the combinatorial metric;
	\item[(v)] $g$ is a Morse isometry with respect to the combinatorial metric.
	\item[(vi)] $g$ skewers a pair of $L$-separated hyperplanes for some $L \geq 0$, i.e., there exist two halfspaces $A \subset B$ such that $g \cdot B \subsetneq A$ and such that there exist at most $L$ hyperplanes transverse simultaneously to the two hyperplanes bounding $A$ and $B$.
\end{itemize}
\end{prop}

\noindent
Recall that, given a metric space $M$ and an isometry $g \in \mathrm{Isom}(M)$, one says that
\begin{itemize}
	\item $g$ is a \emph{Morse isometry} if there exists some $x \in M$ such that $n \mapsto g^n \cdot x$ is a quasi-isometric embedding and if, for every $A,B >0$, there exists some $K \geq 0$ such that any $(A,B)$-quasigeodesic between two points of $\langle g \rangle \cdot x$ stays in the $K$-neighborhood of $\langle g \rangle \cdot x$.
	\item $g$ is a \emph{contracting isometry} if there exists some $x \in M$ such that $n \mapsto g^n \cdot x$ is a quasi-isometric embedding and if there exists some $D \geq 0$ such that the nearest-point projection of any ball disjoint from $\langle g \rangle \cdot x$ onto $\langle g \rangle \cdot x$ has diameter at most~$D$. 
\end{itemize}

\begin{proof}[Proof of Proposition \ref{prop:contracting}.]
The equivalences $(i) \Leftrightarrow (ii) \Leftrightarrow (iii)$ are respectively proved by \cite[Theorem 5.4]{RankContracting} and \cite[Theorem 2.14]{MR3339446}. The equivalence $(iii) \Leftrightarrow (v)$ is clear because the two metrics are quasi-isometric. Finally, the equivalences $(iv) \Leftrightarrow (v) \Leftrightarrow (vi)$ follow respectively from \cite[Lemma 4.6]{MoiHypCube} and \cite[Theorem 4.2]{MR3339446}.
\end{proof}

\subsection{Wallspaces and orientations}

\noindent
Given a set $X$, a \emph{wall} $\{A,B\}$ is a partition of $X$ into two non-empty subsets $A,B$, referred to as \emph{halfspaces}. Two points of $X$ are \emph{separated} by a wall if they belong to two distinct subsets of the partition. 

\medskip \noindent
A \emph{wallspace} $(X, \mathcal{W})$ is the data of a set $X$ and a collection of walls $\mathcal{W}$ such that any two points are separated by only finitely many walls. Such a space is naturally endowed with the pseudo-metric
$$d : (x,y) \mapsto \text{number of walls separating $x$ and $y$}.$$
As shown in \cite{ChatterjiNiblo, NicaCubulation}, there is a natural CAT(0) cube complex associated to any wallspace. More precisely, given a wallspace $(X, \mathcal{W})$, define an \emph{orientation} $\sigma$ as a collection of halfspaces such that:
\begin{itemize}
	\item for every $\{A,B\} \in \mathcal{W}$, $\sigma$ contains exactly one subset among $\{A,B\}$;
	\item if $A$ and $B$ are two halfspaces satisfying $A \subset B$, then $A \in \sigma$ implies $B \in \sigma$.
\end{itemize}
Roughly speaking, an orientation is a coherent choice of a halfspace in each wall. As an example, if $x \in X$, then the set of halfspaces containing $x$ defines an orientation. Such an orientation is referred to as a \emph{principal orientation}. Notice that, because any two points of $X$ are separated by only finitely many walls, two principal orientations are always \emph{commensurable}, ie., their symmetric difference is finite.
\begin{figure}
\begin{center}
\includegraphics[trim={0 13cm 27cm 0},clip,scale=0.4]{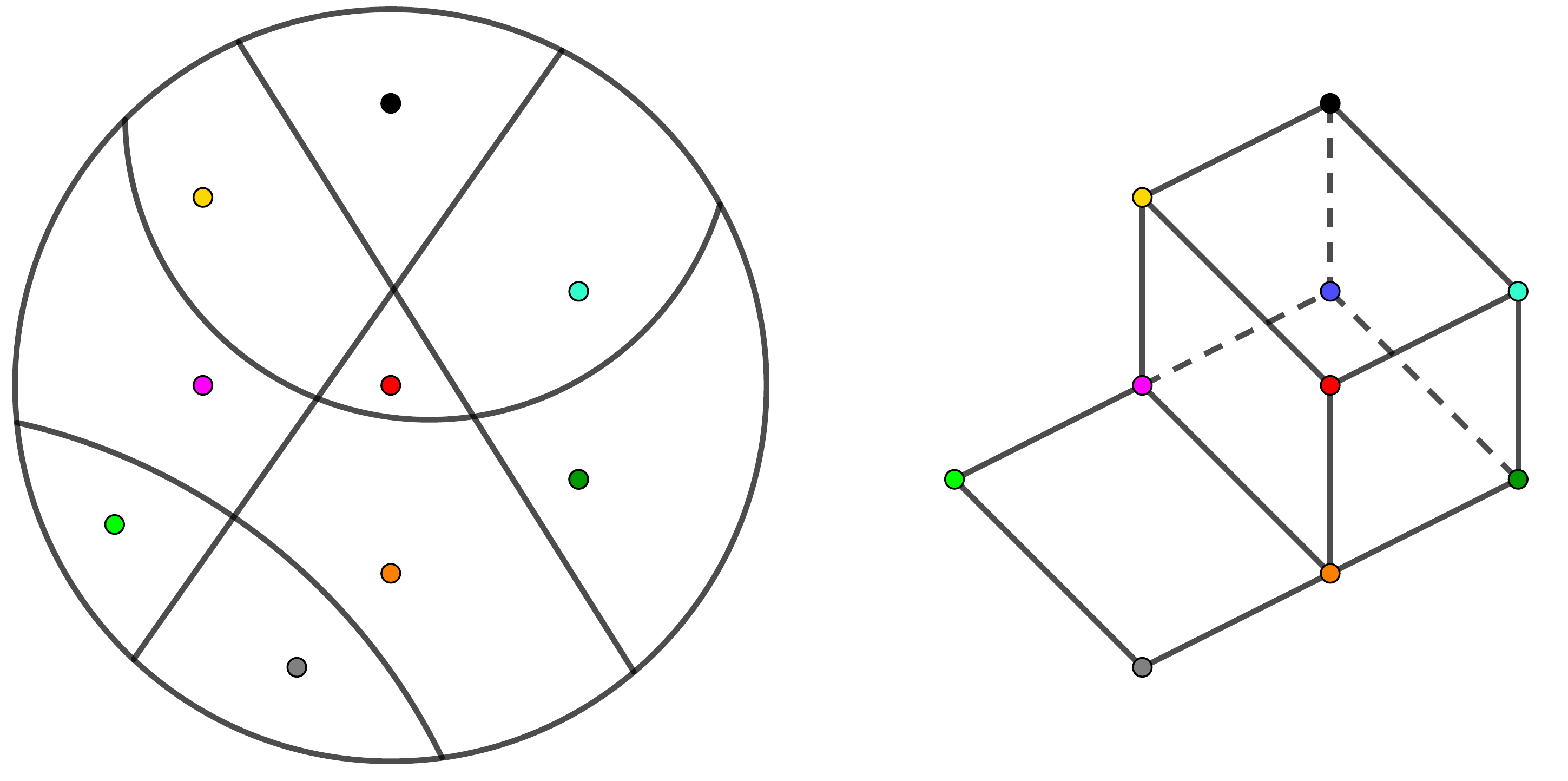}
\caption{A wallspace and its cubulation.}
\label{quotient}
\end{center}
\end{figure}

\medskip \noindent
The \emph{cubulation} of $(X, \mathcal{W})$ is the cube complex
\begin{itemize}
	\item whose vertices are the orientations within the commensurability class of principal orientations;
	\item whose edges link two orientations if their symmetric difference has cardinality two;
	\item whose $n$-cubes fill in all the subgraphs isomorphic to one-skeleta of $n$-cubes.
\end{itemize}
See Figure \ref{quotient} for an example.

\subsection{Roller boundary}

\noindent
Let $X$ be a CAT(0) cube complex. An \emph{orientation} of $X$ is an orientation of the wallspace $(X, \mathcal{W}(\mathcal{J}))$, as defined in the previous section, where $\mathcal{J}$ is the set of all the hyperplanes of $X$. The \emph{Roller compactification} $\overline{X}$ of $X$ is the set of the orientations of $X$. Usually, we identify $X$ with the image of the embedding
$$\left\{ \begin{array}{ccc} X & \to & \overline{X} \\ x & \mapsto & \text{principal orientation defined by $x$} \end{array} \right.$$
and we define the \emph{Roller boundary} of $X$ by $\mathfrak{R}X:= \overline{X} \backslash X$. 

\medskip \noindent
The Roller compactification is naturally a cube complex. Indeed, if we declare that two orientations are linked by an edge if their symmetric difference has cardinality two and if we declare that any subgraph isomorphic to the one-skeleton of an $n$-cube is filled in by an $n$-cube for every $n \geq 2$, then $\overline{X}$ is a disjoint union of CAT(0) cube complexes. Each such component is referred to as a \emph{cubical component} of $\overline{X}$. See Figure \ref{Roller} for an example. Notice that the distance (possibly infinite) between two vertices of $\overline{X}$ coincides with the number of hyperplanes which separate them, if we say that a hyperplane $J$ separates two orientations when they contain different halfspaces delimited by $J$. 
\begin{figure}
\begin{center}
\includegraphics[trim={0 12cm 41cm 0},clip,scale=0.45]{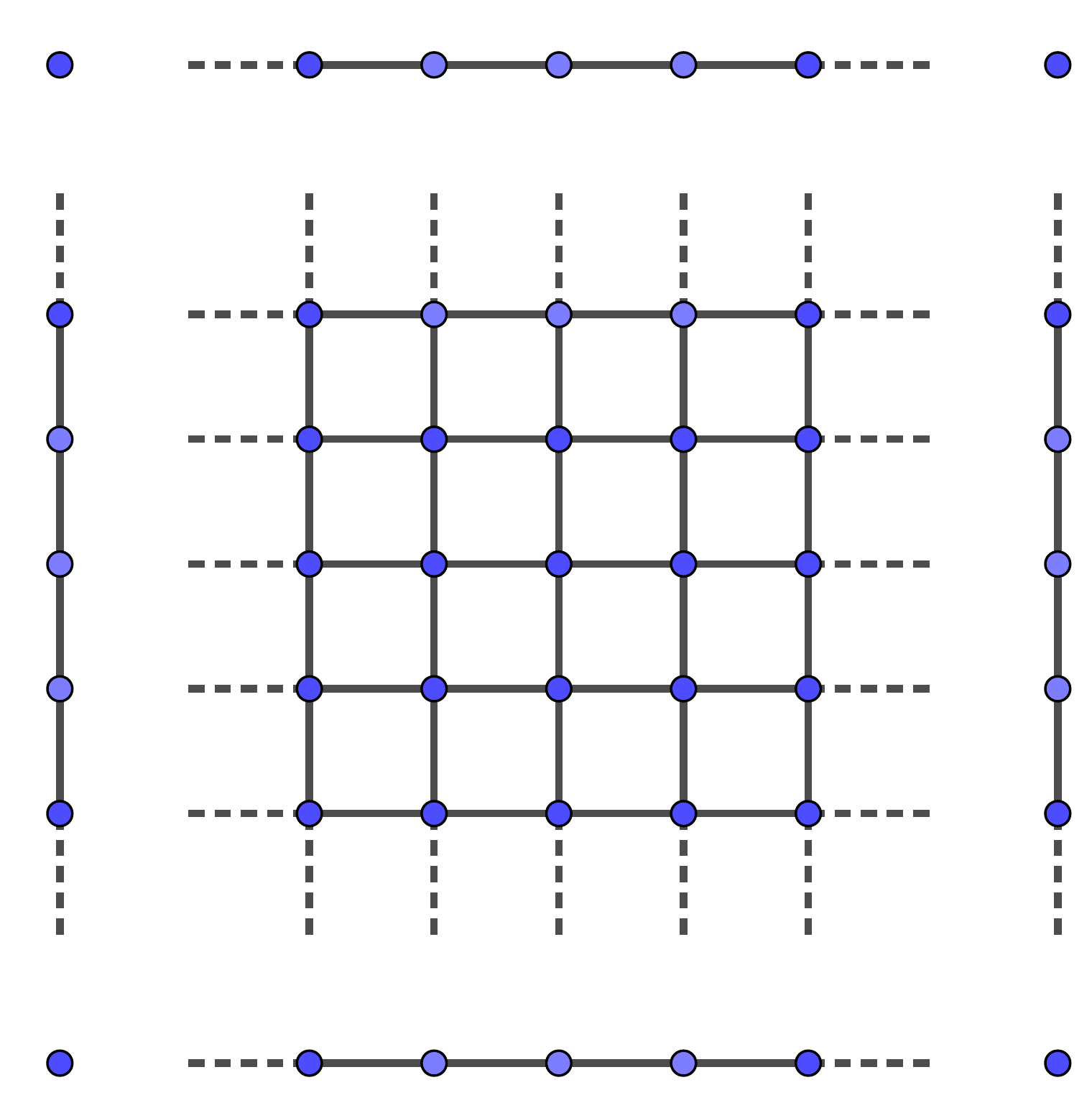}
\caption{Roller compactification of $\mathbb{R}^2$. It contains nine cubical components.}
\label{Roller}
\end{center}
\end{figure}
Two orientations belong to a common cubical component if and only if they differ only on finitely many hyperplanes. A hyperplane separates two orientations if they differ on it.

\medskip \noindent
Interestingly, the projection of a vertex in a finite-dimensional CAT(0) cube complex onto a cubical component of its Roller boundary can be defined:

\begin{prop}\label{prop:ProjInf}
Let $X$ be a finite-dimensional CAT(0) cube complex, $x \in X$ a vertex and $Y \subset \mathfrak{R}X$ a cubical component. There exists a unique point $\xi \in Y$ such that the hyperplanes separating $x$ from $\xi$ are precisely the hyperplanes separating $x$ from $Y$.
\end{prop}

\noindent
In the sequel, the point $\xi$ will be referred to as the \emph{projection} of $x$ onto $Y$. Before turning to the proof of Proposition \ref{prop:ProjInf}, we begin by proving a preliminary lemma.

\begin{lemma}\label{lem:DecreasingInf}
Let $X$ be a CAT(0) cube complex, $\xi \in \mathfrak{R}X$ a point at infinity and $(D_i)$ a decreasing sequence of halfspaces containing $\xi$. For every $i \geq 0$, $D_i$ contains the cubical component of $\xi$.
\end{lemma}

\begin{proof}
Let $Y \subset \mathfrak{R}X$ denote the cubical component containing $\xi$ and, for every $i \geq 0$, let $J_i$ denote the hyperplane delimiting $D_i$. Assume for contradiction that there exists some $k \geq 0$ such that $J_k$ separates at least two points of $Y$. Fix a point $\zeta \in Y$ such that $J$ separates $\zeta$ from $\xi$. Notice that $\zeta \in D_k^c \subset D_i^c$ and $\xi \in D_i$ for every $i \geq k$. Consequently, the hyperplanes $J_k,J_{k+1}, \ldots$ all separate $\zeta$ and $\xi$. But this is impossible because the fact that $\zeta$ and $\xi$ both belong to the cubical component $Y$ implies that they are separated by only finitely many hyperplanes of $X$. 
\end{proof}

\begin{proof}[Proof of Proposition \ref{prop:ProjInf}.]
Let $\sigma$ be the collection of halfspaces defined as follows. If $J$ is a hyperplane which separates two points of $Y$, then $\sigma$ contains the halfspace delimited by $J$ containing $x$ and not its complement. Otherwise, if $J$ does not separate two points of $Y$, then $\sigma$ contains the halfspace delimited by $J$ containing $Y$ and not its complement. Clearly, $\sigma$ is the only possible candidate for our point of $Y$. Now, we need to verify that $\sigma$ is an orientation, and next, that as a point of $\mathfrak{R}X$ it belongs to $Y$.

\medskip \noindent
Let $A$ and $B$ be two halfspaces satisfying $A \subset B$ and $A \in \sigma$. We claim that $B$ belongs to $\sigma$. We distinguish three cases.
\begin{itemize}
	\item If $Y$ is included into $A$, then $Y$ must be contained into $B$ as well, hence $B \in \sigma$.
	\item If the hyperplanes delimiting $A$ and $B$ both separate at least two points of $Y$, then $x$ must belong to $A$, and so to $B$, which implies that $B \in \sigma$.
	\item If the hyperplane delimiting $A$ (resp. $B$) separates (resp. does not separate) two points of $Y$, then $Y$ must be contained into $B$, hence $B \in \sigma$.
\end{itemize}
Thus, we have proved that $\sigma$ is an orientation. Now, assume for contradiction that $\sigma$ does not belong to $Y$. As a consequence, if we fix a point $\xi \in Y$, there exist infinitely many hyperplanes $J_1, J_2, \ldots$ separating $\sigma$ from $\xi$. Because $X$ is finite-dimensional, up to extracting a subspace, we suppose that the $J_i$'s are pairwise disjoint. Moreover, up to re-indexing our sequence, we suppose that $J_i$ separates $J_{i-1}$ and $J_{i+1}$ for every $i \geq 2$. Consequently, if $J_i^+$ denotes the halfspace delimited by $J_i$ which contains $\xi$ for every $i \geq 1$, then $(J_i^+)$ is a decreasing sequence of halfspaces which all contain $\xi$. It follows from Lemma \ref{lem:DecreasingInf} that the $J_i$'s do not cross $Y$, hence $Y \subset J_i^+$ for every $i \geq 1$. Consequently, we must have $J_i^+ \in \sigma$ for every $i \geq 1$ by construction of $\sigma$, a contradiction.
\end{proof}

\noindent
We conclude this subsection by proving a last preliminary lemma. It shows that a cubical component of a uniformly locally finite CAT(0) cube complex must be locally finite. 

\begin{lemma}\label{lem:LocallyFiniteInf}
Let $X$ be a CAT(0) cube complex. Assume that there exists some $N \geq 1$ such that any vertex of $X$ admits at most $N$ neighbors. Then any cubical component of $X$ satisfies the same property. 
\end{lemma}

\begin{proof}
Fix a vertex $x \in \overline{X}$ and let $y_1, \ldots, y_k$ denote its neighbors. For every $1 \leq i \leq n$, let $J_i$ denote the unique hyperplane separating $x$ from $y_i$. 

\medskip \noindent
Fix two distinct indices $1 \leq i , j \leq k$. If the carriers $N(J_i)$ and $N(J_j)$ are disjoint, then it follows from Corollary \ref{cor:HypSepTwoConvex} that there exists a hyperplane $J$ separating $N(J_i)$ and $N(J_j)$. Because $y_i$ belongs the halfspace delimited by $J_i$ which does not contain $J_j$ and that $y_j$ belongs similarly to the halfspace delimited by $J_j$ which does not contain $J_i$, necessarily $J$ separates $y_i$ and $y_j$. But $y_i$ and $y_j$ are within distance two in $\overline{X}$, so that $J_i,J_j,J$ cannot define three distinct hyperplanes separating $y_i$ and $y_j$. Therefore, the carriers $N(J_i)$ and $N(J_j)$ have to intersect.

\medskip \noindent
Because the carriers $N(J_1), \ldots, N(J_k)$ pairwise intersect, according to Helly's property, there exists a vertex $z \in X$ which belongs to the total intersection $\bigcap\limits_{i=1}^k N(J_i)$. By noticing that each $J_i$ defines a distinct edge having $x$ as an endpoint, we conclude that $k \leq N$, as desired.
\end{proof}

\subsection{Median algebras}

\noindent
A \emph{median algebra} $(X,\mu)$ is the data of a set $X$ and a map $\mu : X \times X \times X \to X$ satisfying the following conditions:
\begin{itemize}
	\item $\mu(x,y,y)=y$ for every $x,y \in X$;
	\item $\mu(x,y,z)= \mu(z,x,y)= \mu (x,z,y)$ for every $x,y,z \in X$;
	\item $\mu\left( \mu(x,w,y), w,z \right) = \mu \left( x,w, \mu(y,w,z) \right)$ for every $x,y,z,w \in X$.
\end{itemize}
The \emph{interval} between two points $x,y \in X$ is
$$I(x,y)= \left\{ z \in X \mid \mu(x,y,z)=z \right\};$$
and a subset $Y \subset X$ is \emph{convex} if $I(x,y) \subset Y$ for every $x,y \in Y$. In this article, we are only interested in median algebras whose interval are finite; they are referred to as \emph{discrete} median algebras.

\medskip \noindent
As proved in \cite{NicaCubulation}, a discrete median algebra is naturally a wallspace. Indeed, let us say that $Y \subset X$ is a \emph{halfspace} if $Y$ and $Y^c$ are both convex. Then a \emph{wall} of $X$ is the data of halfspace and its complement, and it turns out that only finitely many walls separate two given point of $X$. The \emph{cubulation} of a discrete median algebra refers to the cubulation of this wallspace. In this specific case, it turns out that any orientation commensurable to a principal orientation must be a principal orientation itself. Consequently, the cubulation of a discrete median algebra $X$ coincides with the cube complex
\begin{itemize}
	\item whose vertex-set is $X$;
	\item whose edges link two points of $X$ if they are separated by a single wall;
	\item whose $n$-cubes fill in every subgraph of the one-skeleton isomorphic to the one-skeleton of an $n$-cube, for every $n \geq 2$. 
\end{itemize}
Therefore, a discrete median algebra may be naturally identified with its cubulation, and so may be thought of as a CAT(0) cube complex. The \emph{dimension} and the \emph{Roller compactification} of a discrete median algebra coincides with the dimension and the Roller compactification of its cubulation. 

\medskip \noindent
Conversely, a CAT(0) cube complex $X$ naturally defines a discrete median algebra (see \cite{mediangraphs} and \cite[Proposition 2.21]{MR2413337}). Indeed, for every triple of vertices $x,y,z \in X$, there exists a unique vertex $\mu(x,y,z) \in X$ satisfying
$$\left\{ \begin{array}{l} d(x,y)= d( x , \mu(x,y,z))+ d(\mu(x,y,z), y) \\ d(x,z)= d(x,\mu(x,y,z))+d(\mu(x,y,z),z) \\ d(y,z)= d(y,\mu(x,y,z))+d(\mu(x,y,z),z) \end{array} \right..$$
Otherwise saying, $I(x,y) \cap I(y,z) \cap I(x,z) = \{ \mu(x,y,z)\}$. The vertex $\mu(x,y,z)$ is referred to as the \emph{median point} of $x,y,z$. Then $(X,\mu)$ is a discrete median algebra, motivating the following terminology:

\begin{definition}
Let $X$ be a CAT(0) cube complex. A \emph{median subalgebra} $Y \subset X$ is a set of vertices stable under the median operation.
\end{definition}

\noindent
We conclude this section by recording a couple of preliminary lemmas. The following statement is precisely what is shown during the proof of \cite[Lemma 2.10]{FTT}.

\begin{lemma}\label{lem:WallHyp}
Let $X$ be a CAT(0) cube complex and $Y \subset X$ a median subset. The walls of $Y$ thought of as a median algebra coincide with the traces of the hyperplanes of~$X$. 
\end{lemma}

\noindent
Our second statement is:

\begin{lemma}\label{lem:ConvexHullInter}
Let $X$ be a CAT(0) cube complex and $\gamma$ a bi-infinite geodesic. Let $\zeta,\xi \in \mathfrak{R}X$ denote the endpoints at infinity of $\gamma$. The union of all the geodesics having their endpoints on $\gamma$ coincides with $I(\zeta,\xi)$. As a consequence, $I(\zeta,\xi)$ is the convex hull of $\gamma$. 
\end{lemma}

\begin{proof}
Let $\sigma$ be a geodesic between two vertices $a^-$ and $a^+$ of $\gamma$. Of course, if $\gamma^\pm$ denotes the subray of $\gamma$ between $a^\pm$ and $\gamma(\pm \infty)$, then the concatenation $\ell = \gamma^- \cup \sigma \cup \gamma^+$ is also a geodesic. As a consequence, a hyperplane of $X$ cannot separate a vertex $x$ of $\sigma$ from $\{\zeta,\xi\}$, because otherwise it would cross $\ell$ twice. Therefore, no hyperplane separates $x$ from the median point $\mu(x,\zeta,\xi)$, which precisely means that $x= \mu(x,\zeta,\xi)$ or equivalently $x \in I(\zeta,\xi)$. 

\medskip \noindent
Conversely, let $x \in I(\zeta,\xi)$ be a vertex. Fix a vertex $y \in \gamma$. Because there exist only finitely many hyperplanes separating $x$ and $y$, there an $n \geq 1$ such that all the hyperplanes which separate $x$ and $y$ and which cross $\gamma$ have to cross $\gamma$ between $\gamma(-n)$ and $\gamma(n)$. We also take $n$ sufficiently large so that $y$ lies between $\gamma(-n)$ and $\gamma(n)$ along $\gamma$. Now, we want to prove that $x$ belongs to a geodesic between $\gamma(-n)$ and $\gamma(n)$, or equivalently that $\mu(x,\gamma(-n),\gamma(-n))=x$. If this equality does not hold, then there must exist a hyperplane $J$ separating $x$ from $\{\gamma(-n),\gamma(n)\}$. Necessarily, $J$ separates $x$ from $y$, so that it follows from our choice of $n$ that $J$ does not cross $\gamma$, i.e., $J$ separates $x$ from $\{\zeta,\xi\}$. But this contradicts the equality $\mu(x,\zeta,\xi)= x$.

\medskip \noindent
Thus, we have proved that $I(\zeta,\xi)$ coincides with the union of all the geodesics having their endpoints on $\gamma$. In order to show the second assertion of our lemma, it remains to show that the interval $I(\zeta,\xi)$ is convex. So let $x,y \in I(\zeta,\xi)$ be two vertices and let $z$ be a vertex of a geodesic $[x,y]$ between $x$ and $y$. As a consequence of what we have just proved, $x$ and $y$ belong to geodesics with endpoints on $\gamma$. Therefore, if $J$ is a hyperplane which does not cross $\gamma$, then $x$ and $y$ have to belong to the same halfspace delimited by $J$ as $\gamma$, and for the same reason the vertex $z$ must belong to this halfspace as well. It follows that no hyperplane separates $z$ from $\gamma$, or alternatively from $\{\zeta,\xi\}$. The conclusion is that $z$ belongs to $I(\zeta,\xi)$. 
\end{proof}

\section{Stable minimising sets of loxodromic isometries}\label{section:SMin}

\noindent
In this section, our goal is to prove the following decomposition theorem about the \emph{stable minimising set} $\mathrm{SMin}(g)= \bigcup\limits_{n \geq 1} \mathrm{Min}(g^n)$ of a loxodromic isometry of a CAT(0) cube complex: 

\begin{thm}\label{thm:Iso}
Let $X$ be a uniformly locally finite CAT(0) cube complex and $g \in \mathrm{Isom}(X)$ a loxodromic isometry. Fix an axis $\gamma$ of $g$, let $\zeta,\xi \in \mathfrak{R}X$ denote its points at infinity, and let $Y \subset \mathfrak{R}X$ be the cubical component containing $\xi$. Then $\mathrm{SMin}(g)$ is a median subalgebra of $X$ and
$$\left\{ \begin{array}{ccc} \mathrm{SMin}(g) & \to & Y \times I(\zeta,\xi) \\ x & \mapsto & \left(\pi_Y( x), \mu(x,\zeta,\xi) \right) \end{array} \right.$$
is an isomorphism of median algebras, where $\pi_Y : X \to Y$ is the projection onto $Y$.
\end{thm}
\begin{figure}
\begin{center}
\includegraphics[trim={0 12cm 35cm 0},clip,scale=0.4]{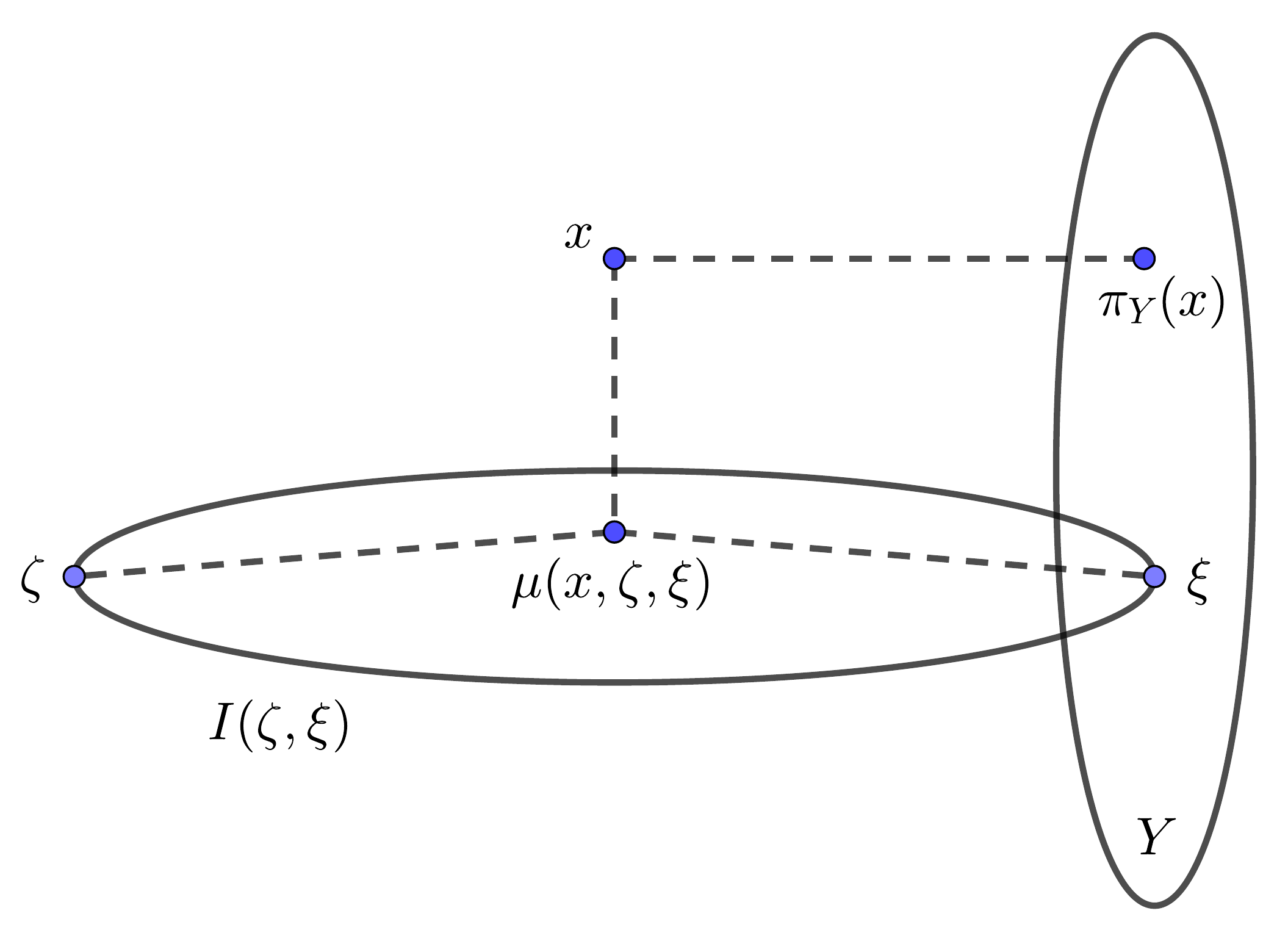}
\caption{Isomorphism $\mathrm{SMin}(g) \to Y \times I(\zeta,\xi)$ from Theorem \ref{thm:Iso}.}
\label{Iso}
\end{center}
\end{figure}

\noindent
The rest of the section is dedicated to the proof of this statement. So let $X$ be a finite-dimensional CAT(0) cube complex and $g \in \mathrm{Isom}(X)$ a loxodromic isometry. Fix an axis $\gamma$ of $g$ and let $\zeta,\xi \in \mathfrak{R}X$ denote its endpoints at infinity. Also, let $Y \subset \mathfrak{R}X$ denote the cubical component which contains $\xi$. Now define the map:
$$\varphi : \left\{ \begin{array}{ccc} X & \to & Y \times I(\zeta,\xi) \\ x & \mapsto & \left( \pi_Y(x), \mu(x,\zeta,\xi) \right) \end{array} \right.,$$
where $\pi_Y : X \to Y$ denotes the projection onto $Y$ as defined by Proposition \ref{prop:ProjInf}. 

\begin{lemma}\label{lem:Iso}
Fix some $n \geq 1$, and let $Q_n$ denote the union of all the axes of $g^n$ having $\zeta,\xi$ as points at infinity. Then $\mathrm{Min}(g^n)$, $\mathrm{Fix}_Y(g^n)$ and $Q_n$ are three median subalgebras of $\overline{X}$ and $\varphi$ induces an isomorphism of median algebras $\mathrm{Min}(g^n) \to \mathrm{Fix}_Y(g^n) \times Q_n$.
\end{lemma}

\begin{proof}
Set 
$$T_n = \left\{ g^{n\infty} \cdot x := \lim\limits_{k \to + \infty} g^{nk} \cdot x \mid x \in \mathrm{Min}(g^n) \right\}.$$ 
According to \cite[Lemmas 4.10 and 4.13]{FTT}, $\mathrm{Min}(g^n)$, $T_n$ and $Q_n$ are median subalgebras of $\overline{X}$ and
$$\left\{ \begin{array}{ccc} \mathrm{Min}(g^n) & \to & T_n \times Q_n \\ x & \mapsto & \left(g^{n\infty} x, \mu(x,\zeta,\xi) \right) \end{array} \right.$$
is an isomorphism of median algebras. First, we notice that this map is induced by $\varphi$.

\begin{claim}\label{claim:gInfty}
For every $x \in \mathrm{Min}(g^n)$, we have $g^{n \infty} \cdot x = \pi_Y(x)$.
\end{claim}

\noindent
If there exists some $x \in \mathrm{Min}(g^n)$ such that $g^{n \infty} \cdot x \neq \pi_Y(x)$, then there exists some hyperplane $J$ separating $x$ from $g^{n \infty} \cdot x$ which crosses $Y$. Let $\alpha \in Y$ be a point such that $J$ separates $\alpha$ and $g^{n \infty}x$. Because $X$ is finite-dimensional and because $J$ crosses an axis of $g^n$, there must exist some $k \geq 1$ such that $g^{kn}J^+ \subsetneq J^+$, where $J^+$ denotes the halfspace delimited by $J$ which contains $g^{n \infty}x$. But then $\{ g^{nkr}J \mid r \geq 1\}$ defines an infinite family of hyperplanes separating $\alpha$ and $g^{n \infty} x$, which is impossible since $g^{n\infty}x$ and $\alpha$ are two points of the same cubical component $Y$. This concludes the proof of our claim.

\medskip \noindent
The next observation required to conclude the proof of our lemma is:

\begin{claim}
We have $T_n=\mathrm{Fix}_Y(g^n)$.
\end{claim}

\noindent
It is clear that $T_n \subset \mathrm{Fix}_Y(g^n)$. Conversely, fix a point $\alpha \in \mathrm{Fix}_Y(g^n)$. If $\alpha = \xi$, there is nothing to prove, so we suppose that $\alpha \neq \xi$. Let $\mathcal{J}$ denote the (non-empty and finite) collection of the hyperplanes separating $\alpha$ and $\xi$. For every $J \in \mathcal{J}$, let $J^+$ denote the halfspace delimited by $J$ which contains $\alpha$. Notice that, because $\alpha$ and $\xi$ are fixed by $g^n$, the intersection $D := \bigcap\limits_{J \in \mathcal{J}} J^+$ is $\langle g^n \rangle$-invariant. Therefore, if we fix a vertex $x \in \gamma$ and if we set $y:= \mathrm{proj}_D(x)$, then, because the projection onto $D$ is 1-Lipschitz according to Lemma \ref{lem:HypProjSeparate}, we know that
$$d(x,g^nx) \leq d(y,g^ny) = d( \mathrm{proj}_D(x), \mathrm{proj}_D(g^nx)) \leq d(x,g^nx),$$
hence $y \in \mathrm{Min}(g^n)$. Moreover, as any hyperplane separating $y$ and $g^ny$ separates $x$ and $g^nx$ according to Lemma \ref{lem:HypProjSeparate}, it follows that the hyperplanes separating $x$ and $g^nx$ are exactly the hyperplanes separating $y$ and $g^ny$. As a consequence, a hyperplane separating $x$ and $y$ has to separate $g^nx$ and $g^ny$. We can iterate the argument and show that a hyperplane separating $g^nx$ and $g^ny$ has to separate $g^{2n}x$ and $g^{2n}y$. And so on. The conclusion is that a hyperplane $J$ separating $x$ and $y$ has to separate $g^{n\infty}x$ and $g^{n \infty} y$. Because such a hyperplane necessarily crosses $Y$, it follows from Claim \ref{claim:gInfty} that it cannot separate $x$ and $\xi$ nor $y$ and $g^{n \infty} y$. On the other hand, we know from Lemma \ref{lem:HypProj} that $J$ separates $x$ from $D$, so that $J$ cannot separate $\alpha$ and $g^{n \infty}y$. We conclude that $J$ separates $\{y,\alpha,g^{n \infty}y\}$ and $\{x,\xi\}$. Thus, we have proved that a hyperplane separating $x$ and $y$ does not separate $\alpha$ and $y$. As a consequence, a hyperplane separating $y$ from $\alpha$ has to separate $x$ from $\xi= \pi_Y(x)$, which implies that it separates $y$ from $Y$. So $\alpha= \pi_Y(y)=g^{n \infty} \cdot y \in T_n$ according to Claim \ref{claim:gInfty}. 
\end{proof}

\noindent
As the inclusions $\mathrm{Min}(g^n) \subset \mathrm{Min}(g^m)$, $\mathrm{Fix}_Y(g^n) \subset \mathrm{Fix}_Y(g^m)$ and $Q_n \subset Q_m$ hold for every integers $n,m \geq 1$ such that $n$ divides $m$, we deduce that $\mathrm{SMin}(g)$, $\bigcup\limits_{n \geq 1} \mathrm{Fix}_Y(g^n)$ and $\bigcup\limits_{n \geq 1} Q_n$ are three median subalgebras of $\overline{X}$ and that $\varphi$ induces an isomorphism of median algebras
$$\mathrm{SMin}(g) \to \left( \bigcup\limits_{n \geq 1} \mathrm{Fix}_Y(g^n) \right) \times \left( \bigcup\limits_{n \geq 1} Q_n \right).$$
Theorem \ref{thm:Iso} now follows from the following two equalities.

\begin{claim}\label{claim:StabFix}
If $X$ is uniformly locally finite, then we have $\bigcup\limits_{n \geq 1} \mathrm{Fix}_Y(g^n) = Y$.
\end{claim}

\begin{proof}
According to Lemma \ref{lem:LocallyFiniteInf}, $Y$ is locally finite. Consequently, if $\alpha \in Y$, the fact that $g$ fixes $\xi \in Y$ implies that $g^N$ fixes $\alpha$ for some sufficiently large $N \geq 1$. Therefore, any point of $Y$ is fixed by a non-trivial power of $g$.
\end{proof}

\begin{claim}
We have $\bigcup\limits_{n \geq 1} Q_n = I(\zeta,\xi)$.
\end{claim}

\begin{proof}
The inclusion $\bigcup\limits_{n \geq 1} Q_n \subset I(\zeta,\xi)$ is clear. Conversely, let $z \in I(\zeta,\xi)$ be a vertex. According to Lemma \ref{lem:ConvexHullInter}, there exist two vertices $x,y \in \gamma$ such that $z$ belongs to a geodesic $[x,y]$ between $x$ and $y$. Fix a sufficiently large integer $N \geq 1$ so that $y$ separates $x$ and $g^Nx$ along $\gamma$. Then, for any choice of a geodesic $[y,g^Nx]$ between $y$ and $g^Nx$, the concatenation
$$\bigcup\limits_{k \in \mathbb{Z}} g^{kN} \cdot \left( [x,y] \cup [y,g^Nx] \right)$$
defines an axis of $g^N$ passing through $z$. So $z \in Q_N$. The reverse inclusion is proved. 
\end{proof}

\begin{remark}
It can be shown that our stable minimising set $\mathrm{SMin}(g)$ is not only median but also convex, and that it coincides with the parallel set $Y_g$ introduced in \cite[Section 3]{CubeUniformGrowth}. We do not include a proof of this observation as it will not be used in the sequel. As it was pointed out to us by Elia Fioravanti, \cite{EffRAAG} also contains relevant information about minimising sets. In particular, alternative proofs of some of our results can be derived from \cite{EffRAAG}. 
\end{remark}

\section{Geometric characterisation of contracting isometries}\label{section:Contracting}

\noindent
In this section, we make explicit the connection between stable minimising sets and the property of being contracting. More precisely, we want to prove:

\begin{thm}\label{thm:SMingQL}
Let $X$ be a uniformly locally finite CAT(0) cube complex and $g \in \mathrm{Isom}(X)$ a loxodromic isometry. Then $g$ is a contracting isometry if and only if $\mathrm{SMin}(g)$ is quasi-isometric to a line. 
\end{thm}

\noindent
We already mentioned geometric characterisations of contracting isometries of CAT(0) cube complexes in Section \ref{section:Isom}. The proof of our theorem is based on the next one:

\begin{prop}\label{prop:ContractingFlat}
Let $X$ be a locally finite CAT(0) cube complex and $g \in \mathrm{Isom}(X)$ a loxodromic isometry. Fix an axis $\gamma$ of $g$. Then $g$ is a contracting isometry if and only if $\gamma$ is quasiconvex and if there does not exist an isometric embedding $\mathbb{R} \times [0,+ \infty) \hookrightarrow X$ such that $\mathbb{R} \times \{0\}$ is sent into the convex hull of $\gamma$.
\end{prop}

\noindent
We emphasize that, in this statement, $\mathbb{R} \times [0,+ \infty)$ and $X$ are thought of as CAT(0) cube complexes endowed with their graph metrics. Also, recall that a subspace $Y$ of a CAT(0) cube complex is \emph{quasiconvex} if there exists a constant $R \geq 0$ such that any geodesic between two points of $Y$ stays in the $R$-neighborhood of $Y$. 

\medskip \noindent
Our proposition is essentially contained in \cite{article3}. We include a sketch of proof for reader's convenience.

\begin{proof}[Sketch of proof of Proposition \ref{prop:ContractingFlat}.]
If $g$ is a contracting isometry, then we know from Proposition \ref{prop:contracting} that its axis $\gamma$ has to be quasiconvex. Moreover, if there exists an isometric embedding $\mathbb{R} \times [0,+ \infty) \hookrightarrow X$ such that $\mathbb{R} \times \{0\}$ is sent into the convex hull of $\gamma$, then any two hyperplanes intersecting $\gamma$ are simultaneously transverse to infinitely many hyperplanes. Consequently, $g$ does not skewer a pair of $L$-separated hyperplanes for any $L \geq 0$, contradicting Proposition \ref{prop:contracting}. Conversely, assume that $g$ is not contracting and that $\gamma$ is quasiconvex. As a consequence of Proposition \ref{prop:contracting}, for every $n \geq 0$ the hyperplanes $A_n$ and $B_n$ dual to the edges $[\gamma(-n-1),\gamma(-n)]$ and $[\gamma(n),\gamma(n+1)]$ of $\gamma$ are simultaneously transverse to infinitely many hyperplanes; fix a hyperplane $C_n$ transverse to both $A_n$ and $B_n$ which satisfies $d(N(C_n),N(\gamma)) \geq n$, where $N(\gamma)$ denotes the convex hull of $\gamma$. As a consequence of \cite[Corollary 2.17]{coningoff}, there exists an isometric embedding $R_n : [-a_n,b_n] \times [0,c_n] \hookrightarrow X$ such that $R(0,0)$ stays in a fixed neighborhood of $\gamma(0)$ when $n$ varies, and such that $[-a_n,b_n] \times \{0\}$ is sent into $N(\gamma)$, $\{b_n\} \times [0,c_n]$ into $N(B_n)$, $[-a_n,b_n] \times \{c_n\}$ into $N(C_n)$, and $\{-a_n\} \times [0,c_n]$ into $N(A_n)$. Notice that $a_n,b_n,c_n \to + \infty$. Because $X$ is locally finite, we can extract from $(R_n)$ a subsequence converging to an isometric embedding $\mathbb{R} \times [0,+ \infty) \hookrightarrow X$ such that $\mathbb{R} \times \{0\}$ is sent into the convex hull of $\gamma$.
\end{proof}

\noindent
Now we are ready to prove our theorem.

\begin{proof}[Proof of Theorem \ref{thm:SMingQL}.]
Fix an axis $\gamma$ of $g$, let $\zeta,\xi$ denote the endpoints at infinity of $g$, and let $Y \subset \mathfrak{R}X$ be the cubical component containing $Y$. As a consequence of Theorem~\ref{thm:Iso}, $\mathrm{SMin}(g)$ is a quasi-line if and only if $I(\zeta,\xi)$ is a quasi-line and if $Y$ is bounded. Also, as a consequence of Lemma \ref{lem:ConvexHullInter}, $\gamma$ is quasiconvex if and only if $I(\zeta,\xi)$ is a quasi-line.

\begin{claim}
Assume that $I(\zeta,\xi)$ is a quasi-line. If $Y$ is unbounded, then there exists an isometric embedding $\mathbb{R} \times [0,+ \infty) \hookrightarrow X$ such that $\mathbb{R} \times \{0\}$ is sent into the convex hull of the axis $\gamma$.
\end{claim}

\noindent
Let $(\xi_n)$ be a sequence of points of $Y$ such that $d_Y(\xi, \xi_n) \to + \infty$. Also, fix a vertex $z \in \gamma$, and set $x_n=g^nz$ and $y_n=g^{-n}z$ for every $n \geq 1$. For convenience, we identify the points $(\xi_n,x_n)$, $(\xi_n,y_n)$, $(\xi,x_n)$ and $(\xi,y_n)$ of $Y \times I(\zeta,\xi)$ with the vertices of $\mathrm{SMin}(g)$ given by the isomorphism of Theorem \ref{thm:Iso}. Notice that it follows from Lemma \ref{lem:WallHyp} that the distance between $(\xi_n,x_n)$ and $(\xi,x_n)$ tends to infinity as $n \to + \infty$, because the number of walls in $Y \times I(\zeta,\xi)$ separating these two points tends to infinity as well. Notice also that $(\xi_n,x_n)$ and $(\xi,y_n)$ belong to the interval between $(\xi,x_n)$ and $(\xi_n,y_n)$, and that  $(\xi,x_n)$ and $(\xi_n,y_n)$ belong to the interval between $(\xi_n,x_n)$ and $(\xi,y_n)$. As a consequence of \cite[Lemma 2.110]{Qm}, there exists an isometric embedding $R_n : [0,a_n] \times [0,b_n]\hookrightarrow X$ such that $(0,0)$, $(a_n,0)$, $(0,b_n)$ and $(a_n,b_n)$ are sent respectively to $(\xi,x_n)$, $(\xi,y_n)$, $(\xi_n,x_n)$ and $(\xi_n,y_n)$. Notice that, since $\gamma$ is quasiconvex, there exists some constant $R \geq 0$ (which does not depend on $n$) such that (the image of) $[0,a_n] \times \{0\}$ intersects the ball $B(z,R)$. Because $X$ is locally finite, up to extracting a subsequence, we may suppose without loss of generality that $B(z,r) \cap \mathrm{Im}(D_n)$ is eventually constant for every $r \geq 1$. Therefore, $(R_n)$ converges to an isometric embedding $R : \mathbb{R} \times [0,+ \infty)$ such that $\mathbb{R} \times \{0\}$ is sent into the convex hull of $\gamma$. This concludes the proof of our claim.

\medskip \noindent
Now we are ready to prove our theorem. First, assume that $g$ is a contracting isometry. We know from \cite[Lemma 3.3]{MR3175245} that $\gamma$ has to be quasiconvex, so $I(\zeta,\xi)$ must be a quasi-line. And it follows from the previous claim and from Proposition \ref{prop:ContractingFlat} that $Y$ must be bounded.

\medskip \noindent
Conversely, assume that $g$ is not contracting. If $I(\zeta,\xi)$ is not a quasi-line, there is nothing to prove, so assume also that $I(\zeta,\xi)$ is a quasi-line. As a consequence to Proposition~\ref{prop:ContractingFlat}, there exists an isometric embedding $\mathbb{R} \times [0,+ \infty) \hookrightarrow X$ such that $\mathbb{R} \times \{0\}$ is sent into the convex hull of $\gamma$. For every $n \geq 0$, let $\rho_n$ denote the geodesic ray of $X$ corresponding to the image of $\left( \{0\} \times [0,n] \right) \cup \left( [0,+ \infty) \times \{n\} \right)$. For every $n \geq 0$, let $\alpha_n$ denote the orientation of $X$ containing all the halfspaces of $X$ in those $\rho_n$ eventually lies. Notice that any two rays among the $\alpha_n$'s and the subray of $\gamma$ starting from $z$ and pointing to $\xi$ cross the same hyperplanes up to finitely many exceptions. As a consequence, they all belong to the same cubical component, namely $Y$, which implies that $Y$ must be infinite. As $Y$ is locally finite according to Lemma \ref{lem:LocallyFiniteInf}, we conclude that it must be unbounded. 
\end{proof}

\section{Centralisers of rank-one isometries}\label{section:Main}

\noindent
This section is dedicated to the main result of the article, making explicit the connection between the \emph{stable centraliser} $SC_G(g)= \{ h \in G \mid \exists n \geq 1, [h,g^n]=1\}$ of an element $g$ which belongs to a group $G$ acting geometrically on a CAT(0) cube complex and the property of being contracting. More precisely:

\begin{thm}\label{thm:StableCentraliser}
Let $G$ be a group acting geometrically on a CAT(0) cube complex $X$, and $g \in G$ an infinite-order element. Fix a cubical component $Y \subset \mathfrak{R}X$ which contains an endpoint of an axis of $g$. Then exactly one of the following situations occurs:
\begin{itemize}
	\item $g$ defines a rank-one isometry of $X$;
	\item the stable centraliser $SC_G(g)$ of $g$ is not virtually cyclic;
	\item $\mathrm{Fix}_Y(g^n)$ is finite for every $n \geq 1$ and the sequence $(\mathrm{Fix}_Y(g^n))$ takes infinitely many values. 
\end{itemize}
\end{thm}

\noindent
Before turning to the proof of our theorem, we begin by proving the following lemma:

\begin{lemma}\label{lem:Qstable}
Let $X$ be a CAT(0) cube complex and $g \in \mathrm{Isom}(X)$ an isometry. Assume that $g$ is loxodromic, fix one of its axes $\gamma$ and let $\zeta,\xi$ denote its points at infinity. Also, assume that $J$ and $gJ$ cannot be transverse for any hyperplane $J$. Then the union of all the axes of $g$ having $\zeta,\xi$ as endpoints at infinity coincides with $I(\zeta,\xi)$. 
\end{lemma}

\begin{proof}
Fix a vertex $y \in I(\zeta,\xi)$. We want to prove that there exists an axis of $g$ passing through $y$ and having $\zeta,\xi$ as endpoints. We assume that $d(y,\gamma)=1$, the general case following by induction. So let $x \in \gamma$ be a vertex adjacent to $y$. Because $y$ belongs to $I(\zeta,\xi)$, the hyperplane $J$ separating $x$ and $y$ has to separate $\zeta$ and $\xi$, so that $J$ meets $\gamma$ along an edge $[a,b]$. Up to replacing $g$ with $g^{-1}$, we may suppose without loss of generality that $[a,b]$ is on the left of $x$ (if we endow $\gamma$ with a left-right orientation so that $g$ translates the points of $\gamma$ to the right). Of course, the hyperplane $gJ$ has to intersect the axis $\gamma$ along the edge $g [a,b]$. We distinguish two cases.

\medskip \noindent
First, assume that $g [a,b]$ is included into the subsegment $[x,gx] \subset \gamma$. Notice that $gJ$ does not separate $y$ and $gy$. Indeed, let $D$ denote the halfspace delimited by $gJ$ which contains $gx$. Because $gJ$ separates $gx$ and $gy$, we have $gx \in D$ and $gy \in D^c$. We also know that $gJ$ separates $x$ and $gx$, so that $x \in D^c$. Next, $J$ is the unique hyperplane separating $x$ and $y$, so we deduce from $J \neq gJ$ that $J$ does not separate $x$ and $y$. Consequently, $y$ has to belong to $D^c$ since $x \in D^c$. Thus, we have proved that $y$ and $gy$ both belong do $D^c$, as desired. Notice also that $J$ separates $y$ and $gy$ since it cannot separate $x$ and $gx$ (otherwise it could cross $\gamma$ twice) and it cannot separate $gx$ and $gy$ because $gJ$ is the unique hyperplane separating $gx$ and $gy$. The conclusion is that the hyperplanes separating $y$ and $gy$ are exactly $J$ and the hyperplanes separating $x$ and $gx$ which are different from $gJ$. As a consequence, the number of hyperplanes separating $x$ and $gx$ equals the number of hyperplanes separating $y$ and $gy$, hence $d(x,gx)=d(y,gy)$. We conclude that $y$ belongs to $\mathrm{Min}(g)$.

\medskip \noindent
Second, assume that $g [a,b]$ is included into the subsegment $[b,x] \subset \gamma$. Because $J$ intersects $\gamma$ just once, it has to separate $a$ from $\{b,x,gx\}$. We know that $J$ separates $x$ and $y$, and we know that it cannot separate $gx$ and $gy$ since $gJ$ is the unique hyperplane separating $gx$ and $gy$. Consequently, $J$ separates $\{a,y\}$ and $\{x,gx,gy,b\}$. Next, because $J$ intersects $\gamma$ just once, it has to separate $\{a,b\}$ and $\{x,gx \}$. We know that $gJ$ separates $gx$ and $gy$, and we know that it cannot separate $x$ and $y$ since $J$ is the unique hyperplane separating $x$ and $y$. Consequently, $gJ$ separates $\{a,b, gy\}$ and $\{x,gx\}$. It follows that $J$ and $gJ$ are transverse, which is impossible.

\medskip \noindent
So far, we have proved that $y$ belongs to $\mathrm{Min}(g)$. It remains to show that $g^\infty y =\xi$ and $g^{-\infty}y=\zeta$. If there exists a hyperplane $J$ separating $g^\infty y$ and $g^\infty x$, then such a hyperplane has to separate $g^n y$ and $g^nx$ for some sufficiently large $n \geq 1$. Up to translating $J$ by $g^{-n}$, we may suppose without loss of generality that $J$ separates $x$ and $y$. On the other hand, according to \cite[Lemma 4.11]{FTT}, the fact that $J$ separates $g^\infty x$ and $g^\infty y$ implies that $J$ does not cross the axis $\gamma$. Therefore, $J$ has to separate $y$ from $\{\zeta,\xi\}$, contradicting the fact that $y$ belongs to $I(\zeta,\xi)$. We conclude that $g^\infty y = g^\infty x = \xi$. One shows similarly that $g^{- \infty}y= \zeta$. 
\end{proof}

\begin{proof}[Proof of Theorem \ref{thm:StableCentraliser}.]
Up to subdividing $X$, we may suppose without loss of generality that $g$ is loxodromic. If $g$ is a contracting isometry, then it follows for instance from the combination of \cite[Theorems 1.3 and 1.5]{SistoContracting}, \cite[Theorem 1.4]{OsinAcyl} and \cite[Corollary 6.6]{DGO} that its stable centraliser is virtually cyclic (although a direct proof is possible). Conversely, assume that $g$ is not contracting. According to Theorem \ref{thm:SMingQL}, $\mathrm{SMin}(g)$ is not a quasi-line. Let $Y \times I(\zeta,\xi)$ be the decomposition of $\mathrm{SMin}(g)$ given by Theorem \ref{thm:Iso}. We distinguish two cases.

\medskip \noindent
Case 1: $Y$ is unbounded. If $\mathrm{Fix}_Y(g^n)$ is bounded for every $n \geq 1$, then the sequence $(\mathrm{Fix}_Y(g^n))$ has to take infinitely many values as $\bigcup\limits_{n \geq 1} \mathrm{Fix}_Y(g^n)=Y$ according to Claim \ref{claim:StabFix}. So assume that $\mathrm{Fix}_Y(g^n)$ is unbounded for some $n \geq 1$. Because we know from Lemma \ref{lem:centraliserCC} that $C_G(g^n)$ acts geometrically on $\mathrm{Min}(g^n)$, which is isomorphic to $\mathrm{Fix}_Y(g^n) \times Q_n$ according to Lemma \ref{lem:Iso}, it follows $C_G(g^n)$ cannot be virtually cyclic. Consequently, $SC_G(g)$ cannot be virtually cyclic either. 

\medskip \noindent
Case 2: $Y$ is bounded. Fix an integer $n \geq 1$ such that $J$ and $g^nJ$ cannot be transverse for any hyperplane $J$ of $X$. (Such an integer exists for instance according to \cite[Lemma 2.2]{HaettelArtin}.) Let $\mathrm{Fix}_Y(g^n) \times Q_n$ be the decomposition of $\mathrm{Min}(g^n)$ given by Lemma~\ref{lem:Iso}. We know from \ref{lem:centraliserCC} that $C_G(g^n)$ acts geometrically on this median subalgebra, and it follows from Lemma \ref{lem:Qstable} that $Q_n= I(\zeta,\xi)$. Because Lemma \ref{lem:LocallyFiniteInf} implies that $Y$ must be finite, it follows that $C_G(g^n)$ contains a finite-index subgroup $C$ which acts geometrically on $I(\zeta,\xi)$. But, because $Y \times I(\zeta,\xi)$ is not a quasi-line and $Y$ is bounded, necessarily $I(\zeta,\xi)$ cannot be quasi-line, so that $C$, and a fortiori $C_G(g^n)$ and $SC_G(g)$, cannot be virtually cyclic. 
\end{proof}

\begin{remark}
Interestingly, the arguments above show that, if an isometry $g$ admits an axis $\gamma$ which is not quasiconvex, then its stable centraliser is not virtually cyclic. Indeed, it follows from Lemma \ref{lem:Qstable} that there exists some $n \geq 1$ such that $\mathrm{Min}(g^n)$ contains $I(\zeta,\xi)$. But the centraliser of $g^n$ acts geometrically on $\mathrm{Min}(g^n)$ according to Lemma \ref{lem:centraliserCC} and we know from Lemma \ref{lem:ConvexHullInter} that the interval $I(\zeta,\xi)$ is not a quasi-line if $\gamma$ is not quasiconvex. A consequence of this observation is that, if our cube complex decomposes as the Cartesian product of unbounded complexes and if the stable centraliser of our isometry $g$ is virtually cyclic, then $g$ has to preserve one of the factors (up to finite Hausdorff distance). This explains why, in Rattaggi's example (described in the introduction), the isometry of the product of trees stabilises a factor. 
\end{remark}

\section{Applications}\label{section:Applications}

\subsection{Special cube complexes}\label{section:Special}

\noindent
As a first application of Theorem \ref{thm:StableCentraliser}, we prove that:

\begin{thm}\label{thm:Special}
Let $G$ be a group acting geometrically on a CAT(0) cube complex $X$. Assume that, for every hyperplane $J$ and every element $g \in G$, the hyperplanes $J$ and $gJ$ are neither transverse nor tangent. Then an infinite-order element $g \in G$ defines a rank-one isometry of $X$ if and only if its stable centraliser $SC_G(g)$ is virtually cyclic. 
\end{thm}

\begin{proof}
Fix an axis $\gamma$ of $g$ and let $Y \subset \mathfrak{R}X$ be a cubical component containing an endpoint of $\gamma$. We claim that $\mathrm{Fix}_Y(g^n)=Y$ for every $n \geq 1$. So let $n \geq 1$ be an integer and let $\zeta \in \mathrm{Fix}_Y(g^n)$ be a point.

\medskip \noindent
If $\xi \in Y$ is a point adjacent to $\zeta$, then there exists a unique hyperplane $J$ which separates them. Of course, $g^n \xi$ must be adjacent to $\zeta$ as well since $g^n$ fixes $\zeta$, so that $g^nJ$ is the unique hyperplane separating $\zeta$ and $g^n\xi$. If $g^nJ \neq J$ then $J$ and $g^nJ$ are the unique hyperplanes separating $\xi$ and $g^n \xi$. Notice that, if $N(J)$ and $N(g^nJ)$ are disjoint, then it follows from Corollary \ref{cor:HypSepTwoConvex} that there exists a hyperplane $H$ separating them. But such a hyperplane would separate $\xi$ and $g^n \xi$, which is impossible. Therefore, $J$ and $g^nJ$ must be either transverse or tangent. Because such a configuration is forbidden by assumption, it follows that $\xi= g^n\xi$. 

\medskip \noindent
Thus, we have proved that $g^n$ fixes all the neighbors of $\zeta$. By arguing by induction over the distance to $\zeta$, we deduce that $g^n$ fixes $Y$ entirely. Now, the desired conclusion follows from Theorem \ref{thm:StableCentraliser}. 
\end{proof}

\noindent
As a particular case of Theorem \ref{thm:Special}, one gets:

\begin{cor}
Let $X$ be a compact special cube complex. A non-trivial element $g \in \pi_1(X)$ defines a rank-one isometry of $\widetilde{X}$ if and only if its centraliser in $\pi_1(X)$ is cyclic.
\end{cor}

\begin{proof}
Because $X$ does not contain self-intersecting or self-osculating hyperplanes, it follows that the action of $\pi_1(X)$ on the universal cover $\widetilde{X}$ satisfies the assumption of Theorem \ref{thm:Special}. Therefore, $g$ (which has infinite order since $\pi_1(X)$ is torsion-free) defines a rank-one isometry of $\widetilde{X}$ if and only if its stable centraliser is virtually cyclic, or equivalently cyclic since $\pi_1(X)$ is torsion-free. Notice that an element of $\pi_1(X)$ commutes with a power of $g$ if and only it commutes with $g$ itself. Indeed, such a property holds for right-angled Artin groups (as an immediate consequence \cite[Theorem 1.2]{MR634562}) and according to \cite[Theorem 4.2, Lemma 4.3]{Special} the fundamental group of a special cube always embeds into a right-angled Artin group. Therefore, the stable centraliser of $g$ turns out to coincide with its centraliser. 
\end{proof}

\subsection{Some two-dimensional cube complexes}\label{section:DimTwo}

\noindent
Our second application of Theorem \ref{thm:StableCentraliser} is:

\begin{thm}\label{thm:TwoDim}
Let $G$ be a group acting geometrically on a two-dimensional CAT(0) cube complex $X$. Assume that the link of a vertex of $X$ cannot contain an induced copy of $K_{2,3}$ in its one-skeleton. Then an infinite-order element $g \in G$ defines a rank-one isometry of $X$ if and only if its stable centraliser is virtually cyclic.
\end{thm}

\noindent
The theorem will be essentially an immediate consequence of Theorem \ref{thm:StableCentraliser} combined with the following lemma:

\begin{lemma}\label{lem:BoundaryTree}
Let $X$ be a two-dimensional CAT(0) cube complex. Assume that the link of a vertex of $X$ cannot contain an induced copy of $K_{2,3}$ in its one-skeleton. A cubical component of $\mathfrak{R}X$ must be a linear tree.
\end{lemma}

\begin{proof}
Let $Y$ be a cubical component of $X$. As the dimension of $Y$ must be smaller than the dimension of $X$, we know that $Y$ is a tree. It remains to show that a vertex of $Y$ has at most two neighbors. Assume for contradiction that $Y$ contains a vertex $\xi$ with at least three neighbors $\alpha,\beta,\gamma \in Y$. Let $A,B,C$ denote the three hyperplanes separating $\xi$ from $\alpha,\beta,\gamma$ respectively. Notice that any two hyperplanes among $A,B,C$ are disjoint in $X$ as they are not transverse in $Y$.

\medskip \noindent
We claim that the carriers $N(A)$, $N(B)$ and $N(C)$ pairwise intersect.

\medskip \noindent
Indeed, if $N(A)$ and $N(B)$ are disjoint, it follows from Corollary \ref{cor:HypSepTwoConvex} that there exists a hyperplane $J$ separating $A$ and $B$. Because $\alpha$ and $\beta$ differ on $A$ and $B$, necessarily they differ on $J$ as well. But $\alpha$ and $\beta$ are two vertices of $Y$ at distance two apart, so they only differ on two hyperplanes. Consequently, $N(A)$ and $N(B)$ have to intersect. One shows similarly that $N(A)$ and $N(C)$, and $N(B)$ and $N(C)$, also intersect, concluding the proof of our claim.

\medskip \noindent
It follows from Helly's property that the intersection $N(A) \cap N(B) \cap N(C)$ is non-empty. Fix one of its vertices $x$. Notice that, if we fix a hyperplane $J$ separating $x$ from $Y$ (which exists as a consequence of Lemma \ref{lem:DecreasingInf}), then $J$ has to be transverse to $A$, $B$ and $C$. Once again as a consequence of Helly's property, there exists a vertex $z$ which belongs to the intersection $N(J) \cap N(A) \cap N(B) \cap N(C)$. By construction, the subgraph in the link of $z$ generated by the vertices which correspond to the edges adjacent to $z$ and dual to $A,B,C,J$ must be isomorphic to $K_{2,3}$, a contradiction.
\end{proof}

\begin{proof}[Proof of Theorem \ref{thm:TwoDim}.]
Because an elliptic isometry of a linear tree always has order at most two, it follows from Lemma \ref{lem:BoundaryTree} that the third case of the trichotomy provided by Theorem \ref{thm:StableCentraliser} cannot happen. The desired conclusion follows. 
\end{proof}

\subsection{Centralisers of regular elements}

\noindent
Our last application is the following statement:

\begin{thm}\label{thm:Regular}
Let $G$ be a group acting geometrically on a CAT(0) cube complex $X$. Assume that $G$ decomposes as a product of $n \geq 1$ unbounded irreducible CAT(0) cube complexes $X_1 \times \cdots X_n$. If $g \in G$ is a regular element, then $SC_G(g)$ is virtually $\mathbb{Z}^n$. 
\end{thm}

\noindent
We begin by proving a preliminary lemma:

\begin{lemma}\label{lem:BoundedInf}
Let $X$ be a CAT(0) cube complex and $g \in \mathrm{Isom}(X)$ a loxodromic isometry with a fixed axis $\gamma$. If $g$ is a contracting isometry, then the cubical component of $\mathfrak{R}X$ containing $\gamma(+ \infty)$ is bounded.
\end{lemma}

\begin{proof}
Let $Y$ denote the cubical component of $\mathfrak{R}X$ which contains $\gamma(+ \infty)$. As a consequence of Proposition \ref{prop:contracting}, there exist an integer $L \geq 1$ and a sequence of pairwise $L$-separated hyperplanes $J_1,J_2, \ldots$ such that $J_i$ separates $J_{i-1}$ and $J_{i+1}$ for every $i \geq 2$. For every $i \geq 1$, let $J_i^+$ denote the halfspace delimited by $J_i$ which contains $J_{i+1}$; notice that $Y \subset J_i^+$ as a consequence of Lemma \ref{lem:DecreasingInf}. Now fix two vertices $\alpha, \beta \in Y$ and let $\mathcal{J}$ denote the collection of the hyperplanes separating them. 

\medskip \noindent
Let $J \in \mathcal{J}$. Fix two vertices $x,y \in X$ separating by $J$; say that $x$ and $\alpha$ belong to the same halfspace delimited by $J$. Because there exist only finitely many hyperplanes separating a given vertex $z \notin J_1^+$ from $x$ or $y$, it follows that there exists some $i_0 \geq 1$ such that $x,y \notin J_i^+$ for every $i \geq i_0$. As $J$ separates $\{\alpha,x\}$ and $\{\beta,y\}$, and, for every $i \geq i_0$, $J_i$ separates $\{x,y\}$ and $\{\alpha,\beta\}$, we deduce that $J$ and $J_i$ are transverse.

\medskip \noindent
Thus, we have proved that any hyperplane of $\mathcal{J}$ is transverse to all but finitely many $J_1,J_2, \ldots$, which implies that $\mathcal{J}$ has cardinality at most $L$. Therefore, $d(x,y) = \# \mathcal{J} \leq L$. The conclusion is that $Y$ has diameter at most $L$. 
\end{proof}

\noindent
We are now ready to prove our theorem.

\begin{proof}[Proof of Theorem \ref{thm:Regular}.]
Let $\gamma$ be an axis of $g$. For every $1 \leq i \leq n$, let $\gamma_i$ denote the projection of $\gamma$ onto $X_i$. Then $\gamma_i$ is a bi-infinite geodesic of $X_i$ on which $g$ acts by translations, so it is an axis of $g$ with respect to the induced action $\langle g \rangle \curvearrowright X_i$. Let $\zeta_i, \xi_i \in \mathfrak{R}X$ denote the endpoints at infinity of $\gamma_i$, and $\zeta, \xi \in \mathfrak{R}X$ the endpoints at infinity of $\gamma$. Also, let $Y_i$ denote the cubical component of $\mathfrak{R}X_i$ which contains $\xi_i$, and $Y$ the cubical component of $\mathfrak{R}X$ which contains $\xi$. Notice that $Y= Y_1 \times \cdots \times Y_n$, so that it follows from Lemma \ref{lem:BoundedInf} that $Y$ is bounded, and in fact finite as a consequence of Lemma \ref{lem:LocallyFiniteInf}.  

\medskip \noindent
We claim that, if $n \geq 1$ is an integer such that $J$ and $g^nJ$ cannot be transverse for any hyperplane $J$ of $X$, then $C_G(g^n)$ is virtually $\mathbb{Z}^n$. (Such an integer exists for instance according to \cite[Lemma 2.2]{HaettelArtin}.) 

\medskip \noindent
Let $\mathrm{Fix}_Y(g^n) \times Q_n$ be the decomposition of $\mathrm{Min}(g^n)$ given by Lemma \ref{lem:Iso}. We know from \ref{lem:centraliserCC} that $C_G(g^n)$ acts geometrically on this median subalgebra, and it follows from Lemma \ref{lem:Qstable} that $Q_n= I(\zeta,\xi)$. Because $Y$ is finite, it follows that $C_G(g^n)$ contains a finite-index subgroup $C$ which acts geometrically on $I(\zeta,\xi)$. On the other hand, we have
$$I(\zeta,\xi) = \prod\limits_{i=1}^n I(\zeta_i,\xi_i) = \prod\limits_{i=1}^n \mathrm{ConvexHull}(\gamma_i),$$
where the last equality is justified by Lemma \ref{lem:ConvexHullInter}. Moreover, we know from Proposition~\ref{prop:contracting} that $\gamma_i$ is a Morse geodesic, so that we deduce from Lemma \ref{lem:ConvexHullInter} that the convex hull of $\gamma_i$ has to stay in a neighborhood of $\gamma_i$. In other words, the interval $I(\zeta,\xi)$ is a product of $n$ quasi-lines. Therefore, $C_G(g^n)$ has to be virtually $\mathbb{Z}^n$, concluding the proof of our claim.

\medskip \noindent
Because $C_G(g^p)$ is contained into $C_G(g^q)$ for every integers $p,q \geq 1$ such that $p$ divides $q$, it follows that $SC_G(g)$ is a union of subgroups which are all virtually $\mathbb{Z}^n$. But, according to \cite[Theorem II.7.5]{MR1744486}, a non-decreasing union of virtually abelian subgroups in a CAT(0) group must be eventually constant, so we conclude that the stable centraliser $SC_G(g)$ has to be virtually $\mathbb{Z}^n$. 
\end{proof}

\section{Open questions}\label{section:Questions}

\noindent
In Sections \ref{section:Special} and \ref{section:DimTwo}, we have shown that the third case of the trichotomy provided by Theorem \ref{thm:StableCentraliser} does not happen in some cases. It would be interesting to find a similar phenomenon for other non-exotic CAT(0) cube complexes. 

\begin{question}\label{q:CubedManifold}
Let $M$ be a compact cubed manifold. Is it true that a non-trivial element $g \in \pi_1(M)$ defines a rank-one isometry of the universal cover $\widetilde{M}$ if and only if its stable centraliser is infinite cyclic?
\end{question}

\noindent
Recall that a \emph{cubed manifold} is a manifold which admits a tessellation as a nonpositively curved cube complex. 

\medskip \noindent
Another interesting direction would be extend (a variation of) Theorem \ref{thm:StableCentraliser} to CAT(0) groups. As a particular case of interest: 

\begin{question}
Let $M$ be a compact Riemannian manifold of nonpositive curvature. Is it true that a non-trivial element $g \in \pi_1(M)$ defines a rank-one isometry of the universal cover $\widetilde{M}$ if and only if its stable centraliser is infinite cyclic?
\end{question}

\noindent
Let us conclude this article with a discussion about the following famous open question: 

\begin{question}\label{question:Famous}
Let $G$ be a group acting geometrically on a CAT(0) cube complex (or more generally, a CAT(0) space). If $G$ does not contain $\mathbb{Z}^2$, is $G$ necessarily hyperbolic? 
\end{question}

\noindent
It is worth noticing that, if the action of $G$ of its CAT(0) cube complex $X$ is such that the third point of Theorem \ref{thm:StableCentraliser} cannot happen, then the fact that $G$ does not contain $\mathbb{Z}^2$ implies that all the infinite-order elements of $G$ are rank-one isometries of $X$. This observation leads to the following natural question:

\begin{question}\label{question:AllContracting}
Let $G$ be a group acting geometrically on a CAT(0) cube complex $X$ (or more generally, a CAT(0) space). Assume that any infinite-order element of $G$ defines a rank-one isometry of $X$. Is $G$ hyperbolic?
\end{question}

\noindent
(It is not difficult to show that, given a group $G$ acting geometrically on a CAT(0) space $X$, then $G$ is hyperbolic if and only if there exists some $D \geq 0$ such that all the infinite-order elements of $G$ are $D$-contracting isometries of $X$. It makes Question \ref{question:AllContracting} even more natural.)   

\medskip \noindent
For instance, it may be expected that the combination of positive answers to Questions~\ref{q:CubedManifold} and \ref{question:AllContracting} leads to a positive answer of Question \ref{question:Famous} for fundamental groups of compact cubed manifolds, generalising \cite{WeakHyp}.

\addcontentsline{toc}{section}{References}

\bibliographystyle{alpha}
{\footnotesize\bibliography{RankOneCentraliser}}

\end{document}